\documentclass[a4paper,11pt, english]{amsart}

\usepackage{amssymb} 
\usepackage{amsfonts}
\usepackage{amsmath}

\usepackage{amsthm}
\usepackage[utf8]{inputenc}
\usepackage{pgf}
\usepackage{graphicx}
\usepackage{tikz}
\usepackage{pdfpages}
\usepackage[all]{xy}
\usepackage{verbatim,enumerate}
\usepackage{hyperref}
\usetikzlibrary{patterns}
\usepackage[yyyymmdd,hhmmss]{datetime}

\newtheorem{thm}{Theorem}[section]

\newtheorem{prop}[thm]{Proposition}

{\theoremstyle{definition}
\newtheorem{exa}[thm]{Example}
}

\newcommand{\R}{\mathbb{R}}

\newcommand{\Z}{\mathbb{Z}}

\newcommand{\CP}{\mathbb{C}P}

\newcommand{\Q}{\mathbb{Q}}

\newcommand{\C}{\mathbb{C}}
\newcommand{\K}{\mathbb{K}}

\newcommand{\T}{\mathbb{T}}

\renewcommand{\epsilon}{\varepsilon}

\title{Euler characteristic and signature of real semi-stable
  degenerations}
\author{Erwan Brugallé}
\address{Erwan Brugall\'e, Universit\'e de Nantes, Laboratoire de
  Math\'ematiques Jean Leray, 2 rue de la Houssini\`ere, F-44322 Nantes Cedex 3,
France}
\email{erwan.brugalle@math.cnrs.fr}

\begin{document}

\begin{abstract}
  We give a motivic proof of the fact that for  non-singular real tropical complete
  intersections, the Euler characteristic of the real part is equal to the
  signature of the complex part. This has originally been
proved by Itenberg in the case of surfaces in $\CP^3$, and has been
successively generalized  by Bertrand, and  by Bihan and Bertrand.
Our proof, different from the previous
approaches, is an
application
of the
motivic nearby fiber of semi-stable degenerations.
In particular it extends the
original result by Itenberg-Bertrand-Bihan to 
real analytic families admitting a $\Q$-non-singular tropical limit.
 \end{abstract} 

\maketitle

Given a real algebraic variety $X$, we denote by $\C X$ the set of its
complex points, and by $\R X$  the set of its
real points.
Recall that the \emph{signature} $\sigma(M)$
of an even dimensional oriented compact smooth manifold $M$ is
defined to be $0$ if $\dim M=4n +2$, and to be
the signature on the (symmetric) intersection form on
$H_{2n}(M;\Q)$ if $\dim M=4n$. The Euler characteristic
with closed support (or Borel-Moore
Euler characteristic) of a
topological space $M$ is denoted by $\chi^c(M)$.

\section{Statement}

\subsection{Context}
A real algebraic variety $X$ is said to \emph{satisfy
$\chi=\sigma$} if
\[
\chi^c(\R X)=\sigma(\C X).
\]
This definition is motivated by the old empirical observation that an
important proportion of 
known real algebraic varieties satisfy $\chi=\sigma$, in
particular among compact non-singular \emph{$M$-varieties}.
By the Smith-Thom inequality
(see for example \cite[Section 3.2]{Mang17}), any real algebraic variety $X$ satisfies
\[
\sum_i \beta_i(\R X)\le \sum_i \beta_i(\C X),
\]
where $\beta_i$ denotes the $i$-th Betti number with
$\Z/2\Z$-coefficients. Furthermore the difference between the two
hand-sides is even, and we say that $X$ is an $(M-i)$-variety if
\[
 \sum_i \beta_i(\C X)- \sum_i \beta_i(\R X)=2i.
\]
If $X$ is a compact and non-singular $M$-variety, Rokhlin
Congruence
\cite[2.7.1]{DK} asserts that
\[
\chi(\R X)=\sigma(\C X) \mod 16,
\]
yet almost all known examples of  non-singular  compact $M$-varieties
satisfy the stronger inequality
$\chi(\R X)=\sigma(\C X)$. If many examples among $M$-surfaces 
not satisfying this  equality are known (see for example
\cite{DK,Mang17}), all non-singular  compact $M$-varieties of 
dimension at least 3 that I know\footnote{I should precise that I know
very few such high-dimensional $M$-varieties.} satisfy $\chi=\sigma$.
A heuristic explanation of
this fact is the following~:
a large class of known real algebraic varieties is
obtained by gluing ``basic real algebraic varieties''
 satisfying $\chi=\sigma$, a
 property that  is preserved under gluing. This latter statement is the
content of Proposition
\ref{prop:gluing} below, which is an elementary observation relying on deep
results about motivic nearby fibers and limit mixed Hodge structures.
By ``basic real algebraic varieties'', we mean for example real
projective spaces, or more generally real toric varieties.

As  an application of this observation, we prove
Theorem \ref{thm:patch} below:  real algebraic
varieties constructed
out of non-singular tropical varieties  satisfy $\chi=\sigma$.
This
generalizes the 
case of complete intersections,  originally proved by Itenberg \cite{Ite97}
for hypersurfaces of dimension 2, by 
Bertrand \cite{Ber2} for hypersurfaces of higher dimensions, and then by
 Bihan
and Bertrand \cite{BerBih07} for any complete intersections.
 Theorem  \ref{thm:patch} is actually about
potentially non-compact  
algebraic varieties, and  we first  briefly recall the extension of the
signature to all complex algebraic varieties. 

\medskip
Given a field $\K$, the \emph{Grothendieck group} $K_0(Var_\K)$ is the abelian group
generated by isomorphism classes $[X]$ of algebraic varieties over $k$
modulo
the \emph{scissor relation}
\[
[X]=[X\setminus Y]+[Y]
\]
for any closed algebraic subvariety $Y$ of $X$.
Considering the product
\[
[X]\times [Y]=[X\times Y]
\]
turns $K_0(Var_\K)$ into a commutative ring with $0=[\emptyset]$ and $1=[pt]$.
It it is not too difficult to show
(see for example \cite[Lecture 1]{Peters10}) that when $\K=\R$ or
$\C$,  Euler
characteristic with closed support provides a ring morphism
\[
\begin{array}{cccc}
  \chi^c:& K_0(Var_\K)& \longrightarrow &\Z
  \\ & X&\longmapsto & \chi^c(\K X)
  \end{array}.
\]
Such a ring (or group) morphism is called a \emph{motivic invariant}.
It is a much less obvious result that the signature of non-singular
projective complex algebraic varieties extends to a motivic invariant
\[
\sigma: K_0(Var_\C)\to \Z.
\]
This is a consequence of the combination of the two following facts:
\begin{itemize}
\item the signature of a non-singular projective complex
  algebraic manifold is the evaluation at $1$
  of its Hirzebruch genus; this is the Hodge index Theorem, see for
  example \cite[Theorem 6.33]{Voi02};
\item  the Hirzebruch genus  extends to a motivic invariant
  $\chi_y: K_0(Var_\C)\to \Z[y]$; see for example \cite[Remark
  5.6]{PetSte08}, since $\chi_y$ is in its turn a specialization of
  the Hodge-Euler polynomial.
\end{itemize}
Using the scissor relation, it easy to check that, as mentioned
above,  basic real
algebraic varieties satisfy $\chi=\sigma$.

\begin{exa}\label{ex:exa}
  One computes easily that $\chi^c(\R)=-1$. On the other hand,
  one has
  \[
  0=\sigma(\CP^1)=\sigma(\C)+\sigma(pt)=\sigma(\C)+1.
  \]
  Hence we deduce that $\chi^c(\R)=\sigma(\C)=-1$, that is to say $\C$
  satisfies $\chi=\sigma$. Since both $\chi^c$ and $\sigma$ are ring
  morphisms, we deduce that $\chi^c(\R^n)=\sigma(\C^n)=(-1)^n$. In
  other words,
  real affine spaces satisfy $\chi=\sigma$.
  By the scissor relation again, this implies that any real algebraic
  variety with has a stratification by real affine spaces satisfies
  $\chi=\sigma$. In particular, all Grassmannians, and more generally
  all  real flag varieties,  satisfy $\chi=\sigma$.
  The case of real toric varieties, i.e. equipped with the standard real structure induced by
the complex conjugation on $\C^*$, can be handled similarly.
  One has
    \[
    \chi^c(\R^*)=-2=-1-1=\sigma(\C^*)-\sigma(pt).
    \]
    Hence   real torus $(\C^*)^n$
  satisfy $\chi=\sigma$, and so does any real toric varieties since
  it admits a stratification by real tori.

\end{exa}

\subsection{Real approximations of $\Q$-non-singular tropical varieties}
Given a fan $\Delta\subset \R^n$, we denote by $Tor_\C(\Delta)$ the
complex toric variety defined by $\Delta$, see \cite{Ful}. 
In this note we only consider
fans $\Delta$ for which
$Tor_\C(\Delta)$ is non-singular, in which  case  we can also consider the
tropical toric variety $Tor_\T(\Delta)$ defined by $\Delta$, see
\cite{BIMS15,MikRauBook}.
In what follows,
we use the notions defined in \cite{IKMZ19}
of $\Q$-non-singular tropical varieties of 
$Tor_\T(\Delta)$, and of tropical limits
of  non-singular analytic families 
$X\subset \CP^N\times D^*$
of  algebraic subvarieties of $\CP^N$  over
the punctured unit disk $D^*\subset \C$.
%Such a family is said to be
%\emph{proper} if the projection to $D^*$ is proper.
For such a family, we denote by $X_t$ the member of the family
corresponding to $t\in D^*$.

\begin{thm}\label{thm:patch}
Let $\Delta$ be a fan defining a non-singular projective toric variety
$Tor_\C(\Delta)$. Suppose that
$X\subset Tor_\C(\Delta)\times D^*$
is a  non-singular real analytic family 
of  algebraic subvarieties of $Tor_\C(\Delta)$, admitting a $\Q$-non-singular
tropical limit in $Tor_\T(\Delta)$.
Then for any subfan $\delta$ of $\Delta$ and
for any $t_0\in \R D^*$ small enough, the real analytic variety $X_{t_0}\cap Tor_\C(\delta)$ 
satisfies $\chi=\sigma$.
\end{thm}
As mentioned above the case of compact hypersurfaces has been earlier proved
by Itenberg \cite{Ite97} and Bertrand \cite{Ber2}.
This has been generalized  to possibly non-compact complete
intersections by Bertand and Bihan \cite{BerBih07}, where they used 
motivic aspects of $\chi^c$ and $\sigma$ to reduce to the case of
hypersurfaces. The strategy in each of
these three papers is then to separately
compute
$\chi^c$ and $\sigma$, and to check that both numbers coincide.
Alternative proofs using tropical homology
 were proposed in the case of  complete intersections by Arnal
\cite{Arn17} (still by equating two separate computations),
and by Renaudineau and Shaw \cite{RenSha18}. In this latter work, the equality
$\chi^c(\R X_t)=\sigma(\C X_t)$ is a consequence of a much
stronger result: it follows from the existence of a spectral
sequence starting from the tropical homology of the tropical limit of
$X$
and converging
to the homology of $\R X$.
Results from \cite{RenSha18}
have recently been generalized to any 
real analytic family admitting a $\Q$-non-singular tropical limit by
Rau, Renaudineau and Shaw \cite{RauRenSha22}.

Hence in a sense all previous  proofs  of Theorem
\ref{thm:patch} are based on  separate computations of both 
$\chi^c$ and $\sigma$. Our proof of
Theorem \ref{thm:patch} uses a different strategy~: both quantities
satisfy the same  gluing relations under
totally real semi-stable degenerations by Proposition \ref{prop:gluing},
while the tropical non-singularity assumption ensures
that all pieces involved in the gluing satisfy $\chi=\sigma$.

\medskip
Note  that in the case of hypersurfaces,
Itenberg, Bertrand, and 
Bertand and Bihan do not work in the tropical geometry framework,
but in the dual setup  of unimodular subdivisions of polytopes.
It is interesting that their proof also applies to non-regular (or non-convex)
subdivisions, that is to say  \emph{real combinatorial hypersurfaces} (see
\cite{IS2} for a definition) also
satisfy $\chi=\sigma$. 
Our proof of Theorem \ref{thm:patch}, as well as the  proofs by
Arnal, Renaudineau and Shaw, and Rau Renaudineau and Shaw do not seem to extend to  real
combinatorial hypersurfaces.

\subsection*{Acknowledgment}
This  work is
partially supported by the grant TROPICOUNT of Région Pays de la
Loire, and the ANR project ENUMGEOM NR-18-CE40-0009-02.
I am grateful to Ilia Itenberg for drawing my attention to the
classifications of real cubic hypersurfaces of dimension 3 and 4,
which helped me in 
improving the last remark of this note.

\section{Proof}

\subsection{Real semi-stable degenerations}

Let $f:X\to D$ be a proper analytic map from a non-singular complex
algebraic manifold
$X$ to the unit disk $D\subset \C$ such that
\begin{itemize}
\item $X_t$ is a non-singular algebraic manifold for all $t\in D^*$;
  \item $X_0=f^{-1}(0)$ is a reduced algebraic variety with non-singular
    components crossing normally.
\end{itemize}
Such a map is called a \emph{semi-stable degeneration} (of 
$X_t$ for $t\ne 0$).
Denoting by $(E_i)_{i\in J}$ the irreducible components of $X_0$, we
define for $I\subset J$
\[
E_I^\circ =\bigcap_{i\in I}E_i\setminus \bigcup_{j\notin J}E_j.
\]

Suppose now that $X$ is real, and that $f$ is real when $D$ is equipped with
the standard complex conjugation. we say that $f$ is
\emph{totally real} if each  irreducible component of $X_0$ is
real.

\begin{prop}\label{prop:gluing}
  Let $f:X\to D$ be a totally real semi-stable degeneration. Then for
  any $t\in\R D^*$ small enough
  \[
  \sigma(\C X_t)=\sum_{\emptyset \ne I\subset J}2^{|I|-1}\sigma(\C E^o_I)
  \qquad \mbox{and}\qquad
   \chi^{c}(\R X_t)=\sum_{\emptyset \ne I\subset J}2^{|I|-1}\chi^{c}(\R E^o_I).
   \]
   In particular $X_t$ satisfies $\chi=\sigma$ as soon as all
   $E^o_I$'s satisfy $\chi=\sigma$.
\end{prop}
\begin{proof}
The statement about the signature follows from deep results on motivic
nearby fibers and limit mixed Hodge structures.
We refer to \cite[Section 1 and 3]{KatSta16} for a concise exposition
of what is needed here, and to \cite[Section 3]{DenLoe01},
\cite[Section 2]{Bit05}, and  \cite[Section 11]{PetSte08} for more details.
Recall first that the
 Hirzebruch genus
 $\chi_y(Y)$ of a projective non-singular complex algebraic variety
 is the polynomial in $y$ defined by
 \[
 \chi_y(Y)=\sum_{p\ge 0}\left(\sum_{q\ge 0}(-1)^qh^{p,q}(Y) \right) y^p,
 \]
 and that it satisfies $\chi_1(Y)=\sigma(Y)$ by the Hodge index Theorem.
Hirzebruch genus turns out to be a motivic invariant, that is to
say it extends to a ring morphism
 $\chi_y:K_0(Var_\C)\to \Z[y]$. In particular $\sigma(Y)=\chi_1(Y)$
for any complex algebraic variety $Y$.
The motivic nearby fiber of $f$, introduced by Denef and Loeser, is defined by
\[
\psi_f=\sum_{\emptyset \ne I\subset J} [E^o_I](1- [\C])^{|I|-1} \in K_0(Var_\C),
\]
and satisfies the great property  $\chi_y(\psi_f)=\chi_y(X_t)$ for
$t\ne 0$.
So one has
\[
\chi_y(X_t)=\sum_{\emptyset \ne I\subset J} \chi_y(E^o_I)
(2-\chi_y(\CP^1))^{|I|-1}
=\sum_{\emptyset \ne I\subset J} (1+y)^{|I|-1}\chi_y(E^o_I).
\]
The first statement of the proposition is obtained by evaluating this
identity at $y=1$.

  The statement about Euler characteristic follows from the
  observation that $\R X_t$ is the disjoint union of
  coverings  of $\R E^o_I$ of degree
  $2^{|I|-1}$, with $I$ ranging over
  all possible sets $\emptyset \ne I\subset J$.
  Indeed, by assumption  one can 
  locally express $X$ in coordinates  at a point of $\R E^o_I$ as the
 solutions of the equation
  \[
  x_1x_2\cdots x_{|I|}=\alpha \subset \R^n\times\R,
  \]
  where $\alpha$ is the deformation parameter, with $\alpha=0$
  corresponding to the central fiber $X_0$.
  Given a non zero $\alpha$, the corresponding smooth
  fiber of $f$ is locally given in
  $\R^n=\R^{|I|}\times\R^{n-|I|}$
  by the solutions of the above equation.
  This set is homeomorphic to the disjoint union of  $2^{|I|-1}$
  copies of
  $\R_{\ge 0}^{|I|}\times\R^{n-|I|}$, and the result follows.
\end{proof}

\subsection{Proof of Theorem \ref{thm:patch}}\label{sec:patch}
Our proof   combines 
Proposition \ref{prop:gluing} and \cite[Section 4]{IKMZ19} to
reduce to  an
elementary computation for complement of real hyperplane arrangements.

Recall that $\Delta\subset \R^n$ is a fan defining 
a non-singular projective toric variety
$Tor_\C(\Delta)$, and that
$f:X\to  D^*$
is a  non-singular real analytic family 
of  algebraic subvarieties of $Tor_\C(\Delta)$ admitting a $\Q$-non-singular
tropical limit $V$ in $Tor_\T(\Delta)$.
Given torus orbits $\mathcal O_\C$
and $\mathcal O_\T$
of
$Tor_\C(\Delta)$ and $Tor_\T(\Delta)$ 
respectively corresponding to the same  cone of $\Delta$,
the  intersection of $V\cap \mathcal O_\T$ is the tropical limit of
$X\cap (\mathcal O_\C\times D^*)$.
Since both $\sigma$ and $\chi^c$ are motivic invariants,
 it is enough to prove Theorem \ref{thm:patch} in the case when
$\delta=\{0\}$, that is to say when $Tor_\C(\delta)=(\C^*)^n$.
In what follows we define $X^o=X\cap ((\C^*)^n \times D^*)$ and
$V^o=V\cap\R^n$.

By \cite[Proposition 51]{IKMZ19} (see also
\cite[Proposition 2.3]{HelKat12} and \cite[Lemma 7.9]{Katz2} in a
slightly different realm), after shrinking $D$ if necessary, and after
a base change $t\mapsto \pm t^d$, one can extend
the real family $X^o$ over $0$ to
a real semi-stable degeneration  $f:\overline X^o\to D$ such that
%\begin{enumerate}
%\item
any irreducible component of $E_I^\circ$ of the central fiber
$\overline X^o_0$ is isomorphic to
the complement of a hyperplane arrangement in some
    $\CP^m$.
%  \item the  complex dual to  $\overline X^o_0$ is combinatorially
%    isomorphic to a polyhedral subdivision of $V^o$.
%\end{enumerate}
When $d$ is even, the sign in the base change is chosen so that
the real fiber $X_{t_0}$ lifts to a real fiber (i.e. it is the sign of
$t_0$).
Since the complex conjugation on $(\C^*)^n$ induces, via the 
the tropicalization procedure, the identity map on $\R^n$, 
 the semi-stable degeneration  $f:\overline X^o\to D$ is totally
real.
In particular, each irreducible component of $E_I^\circ$
is the complement of a real hyperplane arrangement in some
    $\CP^m$ equipped with the standard complex conjugation.

Hence, thanks to Proposition \ref{prop:gluing}, the proof of Theorem
\ref{thm:patch} reduces to verify that these latter complements  satisfy $\chi=\sigma$.
This is easily done by a double induction on $m$ and the number
$k$ of
hyperplanes in the arrangement.
\begin{enumerate}
\item The case $m=0$ holds trivially.
  \item Assume that this is true for $m-1$ and any $k$, and let us
    prove by
    induction on $k$ that it is also true for $m$ and any $k$.
    \begin{enumerate}
    \item One checks easily that the case $k=0$ holds, see  Example \ref{ex:exa}.
      \item By definition, the complement $\mathcal A$ of $k$ real hyperplanes in
        $\CP^m$ is obtained by removing  the complement $\mathcal A''$
        of $k-1$
        real hyperplanes in $\CP^{m-1}$  to  the
        complement $\mathcal A'$ of $k-1$ real hyperplanes in
        $\CP^{m}$.
        By induction and the scissor relation, we have
        \[
        \chi^c(\R \mathcal A)=\chi^c(\R \mathcal A')-\chi^c(\R \mathcal A'')
        = \sigma^c(\C \mathcal A')-\sigma^c(\C \mathcal A'')= \sigma^c(\C \mathcal A).
        \]
    \end{enumerate}
\end{enumerate}
Now the proof of Theorem
\ref{thm:patch} is complete. \hfill$\square$

\section{Further comments}
We end this note with a couple of remarks.
First,
dropping compactness or  $M$-condition one constructs easily
real algebraic varieties which do not satisfy
$\chi=\sigma$. For example $X=\C^*$ equipped with the real structure
$\tau(z)=\frac{1}{z}$ is an $M$-curve and satisfies
\[
\chi^c(\R X)=\chi(S^1)=0\ne -2= \sigma(\C^*),
\]
and a quadric ellipsoid $X$ in $\CP^3$ satisfies
\[
\chi(\R X)=2\ne 0= \sigma(\C X).
\]
More generally, Gudkov-Kharlamov-Krakhnov Congruence \cite[2.7.1]{DK} implies that a
non-singular compact
$(M-1)$-variety\footnote{Unfortunately, I do not know much more
  $(M-1)$ than $M$-varieties.}
cannot 
satisfy
$\chi=\sigma$.

Lastly,  Viro Conjecture for $M$-surfaces states that
 such a projective and simply connected $M$-surface $X$ satisfies
\[
\beta_1(\R X)\le h^{1,1}(\C X),
\]
which is equivalent in this case to the inequality
\[
\chi^c(\R X)\ge \sigma(\C X).
\]
I do not know any $M$-variety\footnote{Go back to footnote 1.}, projective or not, singular or not,
which does
not satisfy
\begin{align*}
   \chi^c(\R X)\le \sigma(\C X) &\mbox{ if }\dim X = 0,3\mod 4
 \\ \chi^c(\R X)\ge \sigma(\C X) &\mbox{ if } \dim X = 1,2\mod 4.
\end{align*}

\bibliographystyle{alpha}
\bibliography{../Biblio}
\end{document}